\documentclass[10pt,a4paper]{amsart}
\usepackage{graphicx,multirow,array,amsmath,amssymb,epsf}

\newtheorem{theorem}{Theorem}
\newtheorem{lemma}[theorem]{Lemma}

\begin{document}

\title{Upper bound on lattice stick number of knots}
\author[K. Hong]{Kyungpyo Hong}
\address{Department of Mathematics, Korea University, 1, Anam-dong, Sungbuk-ku, Seoul 136-701, Korea}
\email{cguyhbjm@korea.ac.kr}
\author[S. No]{Sungjong No}
\address{Department of Mathematics, Korea University, 1, Anam-dong, Sungbuk-ku, Seoul 136-701, Korea}
\email{blueface@korea.ac.kr}
\author[S. Oh]{Seungsang Oh}
\address{Department of Mathematics, Korea University, Anam-dong, Sungbuk-ku, Seoul 136-701, Korea}
\email{seungsang@korea.ac.kr}

\thanks{This research was supported by Basic Science Research Program through
the National Research Foundation of Korea(NRF) funded by the Ministry of Education,
Science and Technology (No.~2009-0074101).}

\begin{abstract}
The lattice stick number $s_L(K)$ of a knot $K$ is
defined to be the minimal number of straight line segments
required to construct a stick presentation of $K$ in the cubic lattice.
In this paper, we find an upper bound on the lattice stick number of a nontrivial knot $K$,
except trefoil knot, in terms of the minimal crossing number $c(K)$ which is $s_L(K) \leq 3 c(K) +2$.
Moreover if $K$ is a non-alternating prime knot, then $s_L(K) \leq 3 c(K) - 4$.
\end{abstract}

\maketitle

\section{Introduction} \label{sec:intro}
A simple closed curve embedded into the Euclidean 3-space $\mathbb{R}^3$ is called a {\em knot\/}.
Two knots $K$ and $K'$ are said to be
{\em equivalent\/}, if there exists an orientation preserving
ambient isotopy of $\mathbb{R}^3$ which maps $K$ to $K'$,
or to say roughly, we can obtain $K'$ from $K$ by a sequence
of moves without intersecting any strand of the knot.
A knot equivalent to another knot in a plane of the 3-space is said to be {\em trivial\/}.

A {\em stick knot\/} is a knot which consists of finite number of straight line segments, called {\em sticks\/}.
One natural question concerning stick knots may be to determine the number of sticks.
The {\em stick number\/} $s(K)$ of a knot $K$ is defined to be the minimal number of sticks
required to construct this stick knot.
Several upper and lower bounds on the stick number for various classes of knots and links were founded.
Especially Negami found upper and lower bounds on the stick number of any nontrivial knot $K$
in terms of the minimal crossing number $c(K)$ \cite{N};
\[ \frac{5+ \sqrt{25+8(c(K)-2)}}{2} \leq s(K) \leq 2c(K). \]
Later Huh and Oh \cite{HO3} improved Negami's upper bound to $s(K) \leq \frac{3}{2} (c(K)+1)$.
Furthermore $s(K) \leq \frac{3}{2} c(K)$ for a non-alternating prime knot $K$.

In this paper we deal with another quantity concerning stick knots.
A {\em lattice stick knot\/} is a stick knot in the cubic lattice
$\mathbb{Z}^3=(\mathbb{R} \times \mathbb{Z} \times \mathbb{Z}) \cup (\mathbb{Z} \times \mathbb{R}
\times \mathbb{Z}) \cup (\mathbb{Z} \times \mathbb{Z} \times \mathbb{R})$.
The {\em lattice stick number\/} $s_{L}(K)$ of a knot $K$ is defined to be the minimal number of sticks
required to construct this lattice stick knot.
From the definition we easily know that the lattice stick number of trivial knot is $4$.
Janse van Rensburg and Promislow \cite{JaP} proved that $s_L(K) \ge 12$ for any nontrivial knot $K$.
Also Huh and Oh \cite{HO1,HO2} proved that  $s_L(3_1) = 12$, $s_L(4_1) = 14$,
and $s_L(K) \geq 15$ for any other non-trivial knot $K$
via a simple and elementary argument, called {\em proper leveledness\/}.
Lattice stick knots of the knots $3_1$ and $4_1$ are depicted in Figure \ref{fig:1}.

In this paper we establish an upper bound on the lattice stick number of a nontrivial knot $K$
in terms of the minimal crossing number $c(K)$.

\begin{theorem} \label{thm:main}
Let $K$ be a nontrivial knot which is not trefoil knot. Then $s_L(K) \leq 3 c(K) + 2$.
Moreover if $K$ is a non-alternating prime knot, then $s_L(K) \leq 3 c(K) - 4$.
\end{theorem}

\begin{figure}[h]
\begin{center}
\includegraphics[scale=1]{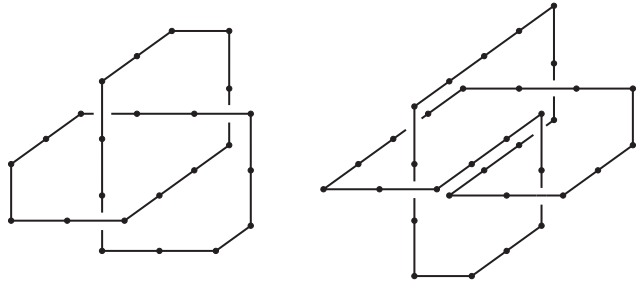}
\end{center}
\vspace{-2mm}
\caption{$3_1$ and $4_1$ in the cubic lattice}
\label{fig:1}
\end{figure}

\section{Arc index and dual arc presentation}

There is an open-book decomposition of $\mathbb{R}^3$ which has open
half-planes as pages and the standard $z$-axis as the binding axis.
We may regard each page as a half-plane $H_{\theta}$ at angle
$\theta$ when the
$x$-$y$ plane has a polar coordinate.
It can be easily shown that every knot $K$ can be embedded in an open-book
with finitely many pages so that it meets each page in a simple arc.
Such an embedding is called an {\em arc presentation\/} of $K$.
And the {\em arc index\/} $a(K)$ is defined to be the minimal number of
pages among all possible arc presentations of $K$.
For example, the leftmost figure in Figure \ref{fig2} shows
an arc presentation of figure-$8$ knot which has the arc index $6$.
Here the points of $K$ on the binding axis are called {\em binding indices\/},
assigned by $1,2, \cdots , a(K)$ from bottom to top.
Also we assign the page numbers $1,2, \cdots , a(K)$ to all the arcs from the back page to the front.

Now we introduce a dual arc presentation.
To distinguish from binding indices we put the related page number with a circle on each arc.
Take a proper line $l$ which does not lie on any plane containing the binding axis.
As in the second figure we isotope each arc to the union of two line segments meeting at the line $l$.
Then all page numbers are appeared on $l$ in the order of increasing or decreasing.
Now rotate it $180^{\circ}$ around the binding axis so that binding indices and
page numbers are exchanged, followed by reversing the isotopy as before.
The result is called a {\em dual arc presentation\/}.
Obviously both presentations represent the same knot.
We mention that in a dual arc presentation the page number $k$ and the two binding indices
$b_1$ and $b_2$ at the end points of the arc with the page number $k$
came from the binding index $k$ and the adjoining two arcs
with the page numbers $b_1$ and $b_2$ at the binding point $b_k$,
respectively, of the original arc presentation.
Refer the bold arc segments in the figure.
This fact is crucial in the last part of the proof of the main theorem.

\begin{figure} [h]
\begin{center}
\includegraphics[scale=1]{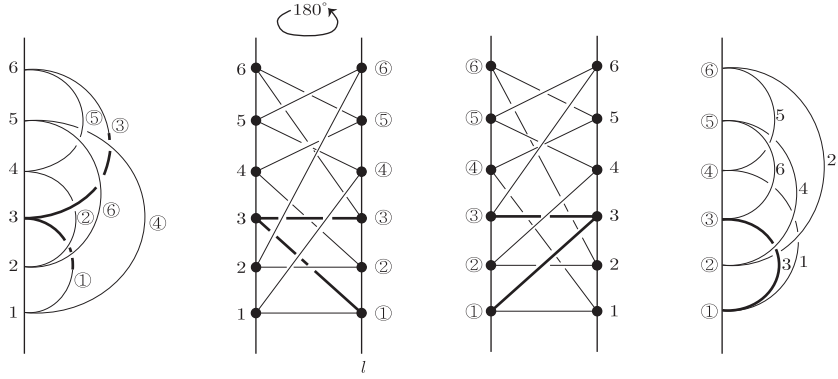}
\end{center}
\vspace{-2mm}
\caption{An arc presentation and its dual arc presentation}
\label{fig2}
\end{figure}

Bae and Park established an upper bound on arc index in terms of
crossing number. Corollary $4$ and Theorem $9$ in \cite{BP} provide the following;

\begin{theorem} {\textup{(Bae-Park)}} \label{thm:BP}
Let $K$ be any nontrivial knot. Then $a(K) \leq c(K)+2$. Moreover
if $K$ is a non-alternating prime knot, then $a(K) \leq c(K)+1$.
\end{theorem}

Later Jin and Park improved the second part of the above theorem
as Theorem $3.3$ in \cite{JiP}.

\begin{theorem} {\textup{(Jin-Park)}} \label{thm:JP}
Let $K$ be any non-alternating prime knot. Then $a(K) \leq c(K)$.
\end{theorem}

\section{Lattice arc presentation and circular arc presentation}

A stick in the cubic lattice which is parallel to the $x$-axis is called an {\em $x$-stick\/}
and similarly define a $y$-stick and a $z$-stick.
We move the binding axis to the line $y=x$ on the $x$-$y$ plane and
replace each arc by two connected sticks which are an $x$-stick and a $y$-stick
properly below the line $y=x$.
For better view we slightly perturb some sticks if they are overlapped.
The resulting is called a {\em lattice arc presentation\/} of the knot on the plane.
See the leftmost figure in Figure \ref{fig3}.

This lattice arc presentations of knots are very useful to construct stick knots
in the cubic lattice.
Let $K$ be a nontrivial knot whose arc index is $a(K)$.
We start with a lattice arc presentation of $K$ with $2 a(K)$ sticks on the $x$-$y$ plane.
Now $b_i$, called a {\em binding point\/}, denotes the point $(i,i)$ on the binding axis
for each binding index $i$.

To construct a lattice stick presentation of $K$ on $\mathbb{Z}^3$, we introduce several notations.
For convenience the notations concerning the $y$ and $z$-coordinates
will be defined in the same manner as the $x$-coordinate.
We denote an $x$-stick
$(x_{ii'},j,k) = \{(x,y,z) \in \mathbb{Z}^3 \, | \, i\leq x \leq i', \, y=j, \, z=k \}$
for some integers $i,i',j$ and $k$.
The plane with the equation $x=n$ for some integer $n$ is called an $x$-{\em level\/} $n$.

First we prove a simple version.

\begin{theorem} \label{thm:-2}
$s_L(K) \leq 3 a(K) - 2$ for a nontrivial knot $K$.
\end{theorem}

\begin{proof}
Starting with a lattice arc presentation of $K$ with $a(K)$ arcs,
we will construct a lattice stick presentation of $K$.
Pick an arc with the page number $k$ and
the binding indices $i$ and $j$ at its endpoints for some $i,j,k = 1,2, \cdots, a(K)$.
Assume that $i<j$.
First we locate two sticks $(x_{ij},i,k)$ and $(j,y_{ij},k)$ on $\mathbb{Z}^3$.
Build all the other $a(K)-1$ arcs in the same way.
Since for every binding point, say $b_i$, there are exactly two arcs attached to,
we can connect this pair of arcs by a $z$-stick $(i,i,z_{kk'})$
where $k$ and $k'$ are the page numbers of the pair.
By connecting all such pair by these $z$-sticks,
we finally get a lattice stick presentation of $K$ with $3 a(K)$ sticks
as the middle figure in Figure \ref{fig3}.

\begin{figure} [h]
\begin{center}
\includegraphics[scale=1]{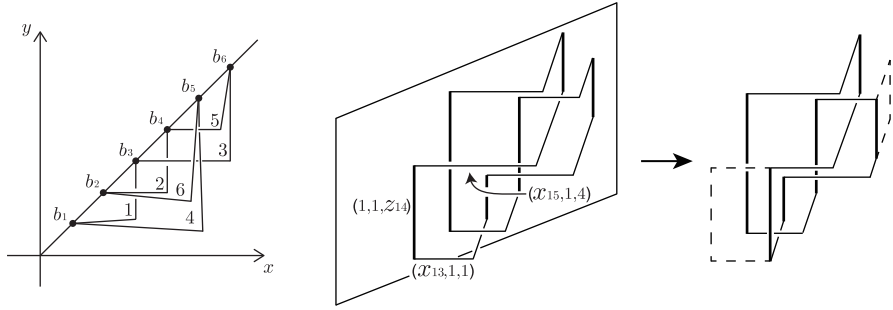}
\end{center}
\vspace{-2mm}
\caption{$s_L(K) \leq 3 a(K) - 2$}
\label{fig3}
\end{figure}

Now we will reduce two sticks from this presentation easily.
Consider the two $x$-sticks on $y$-level $1$ and the $z$-stick connecting them.
These two $x$-sticks are $(x_{1i},1,k)$ and $(x_{1i'},1,k')$.
Without loss of generality we may assume that $i < i'$ and $k < k'$,
so this $z$-stick is $(1,1,z_{kk'})$.
Delete the shorter $x$-stick $(x_{1i},1,k)$, and replace the other two sticks
by an $x$-stick $(x_{ii'},1,k')$ and a $z$-stick $(i,1,z_{kk'})$ to represent $K$ again.
Similarly we repeat this replacement for the two $y$-sticks on $x$-level $a(K)$
and the $z$-stick connecting them.
Then as the last figure we get a lattice stick presentation of $K$ with $3 a(K) - 2$.
\end{proof}

We introduce another expression of an arc presentation of $K$.
First take a $2$-dimensional circular disk
whose boundary actually indicates the binding axis combined with the infinity, say $b_{\infty}$.
Put its binding indices running clockwise on the boundary of this disk starting right after $b_{\infty}$.
Now draw $a(K)$ straight chords for all arcs whose endpoints lie on the boundary.
During this procedure all the page numbers and the related under-over crossings are unchanged.
This expression is called a {\em circular arc presentation\/} of $K$
and the new binding axis is called a {\em circular binding axis\/}.
See the left figure in Figure \ref{fig4} for an example of figure-$8$ knot.

\begin{figure} [h]
\begin{center}
\includegraphics[scale=1]{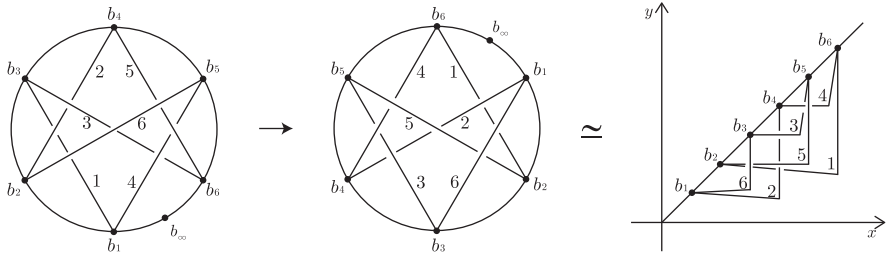}
\end{center}
\vspace{-2mm}
\caption{Circular arc presentation and how rotations work}
\label{fig4}
\end{figure}

We remark the following two facts about arc presentations of knots.
Like turning pages of a book we can rotate the page numbers of an arc presentation,
so that if we turn $m$ number of pages for some integer $m$
then each page number $k$ goes to $k+m$ (mod* $a(K)$).
Here $x$ (mod* $y$) $=$ $x$ (mod $y$) with one exception
so that we use $y$ instead of $0$.
This changing is called {\em page number rotating\/}.
Similarly we can also rotate binding points along the circular binding axis,
so that for some integer $n$ each binding index $i$ goes to $i+n$ (mod* $a(K)$).
This rotation is easily understood when we relocate $b_{\infty}$ on a circular arc presentation of $K$.
This is call {\em binding index rotating\/}.
The middle and right figures in Figure \ref{fig4} explain how these two rotations work.
Clearly both rotations do not change the knot $K$.

An arc presentation of $K$ is called {\em star shaped\/}
if it has an odd number of arc index $a(K)$ and
the difference between the two binding indices of the end points of every arc is always
either $n$ or $n+1$ where $n = \frac{a(K)-1}{2}$.
Let $c_i$, $i=1, \cdots, 2n+1$, be the arc whose binding points are $b_i$ and $b_j$
where $j = i+n$ (mod* $a(K)$).
In Figure \ref{fig5} we give an example of a star shaped arc presentation of a knot.

\begin{figure} [h]
\begin{center}
\includegraphics[scale=1]{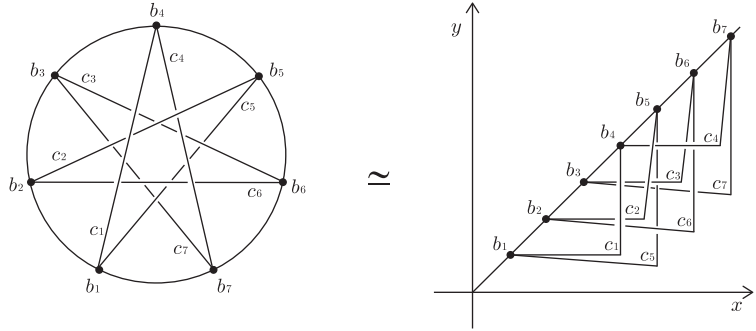}
\end{center}
\vspace{-2mm}
\caption{Star shaped arc presentation}
\label{fig5}
\end{figure}

\section{Proof of Theorem \ref{thm:main}}

From now on $K$ is a nontrivial knot which is not trefoil knot.
For convenience we just use $i$ instead of $i$ (mod* $a(K)$) for a binding index or a page number $i$.
We distinguish into two cases according to the shapes of its arc presentations.

\begin{lemma} \label{lem:nonstar}
If $K$ has a non-star shaped arc presentation with $a(K)$ arcs,
then $s_L(K) \leq 3 a(K) - 4$.
\end{lemma}

\begin{proof}
Note that in a star shaped arc presentation every pair of arcs sharing
one binding index has the property that the difference between the two binding indices
at the other ends of these two arcs is either $1$ or $a(K)-1$.
Thus if $K$ has a non-star shaped arc presentation,
there is a pair of arcs sharing one binding index, say $\beta'$, such that the difference
between the two binding indices, say $\alpha'$ and $\gamma'$,
at the other ends of these two arcs is neither $1$ nor $a(K)-1$.
Without loss of generality $\alpha'$, $\beta'$ and $\gamma'$ appear
in order to clockwise along the circular binding axis.
After a proper binding index rotating these three indices go to $\alpha$, $\beta$ and $a(K)$,
respectively, so that $1 < \alpha < \beta < a(K)$.
Indeed $\alpha = \alpha' + a(K) - \gamma'$ and $\beta = \beta' + a(K) - \gamma'$.
Also we do a page number rotating
so that the arc having binding indices $\alpha$ and $\beta$ has the page number $1$.
Let $k$ be the page number of the other arc having binding indices $\beta$ and $a(K)$.
See the figures in Figure \ref{fig6} for better understanding.

\begin{figure} [h]
\begin{center}
\includegraphics[scale=1]{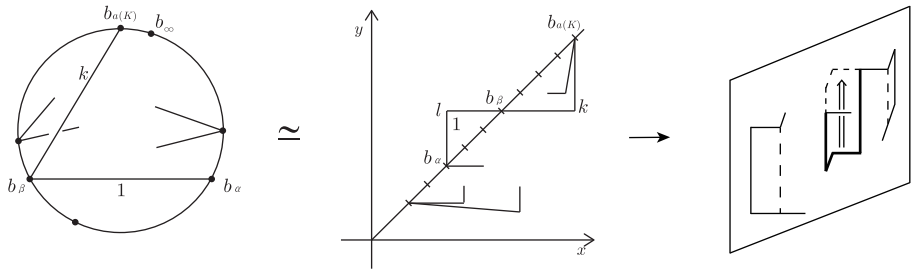}
\end{center}
\vspace{-2mm}
\caption{Lattice stick presentation from a non-star shaped arc presentation}
\label{fig6}
\end{figure}

We slightly change the related lattice arc presentation
so that only the arc, say $l$, with the page number $1$ lies above the line $y=x$.
This does not change the knot $K$ because the arc $l$ does not cross the other part of $K$
when rotating it $180^{\circ}$ around the binding axis.
Now repeat the whole part of the proof of Theorem \ref{thm:-2}.
Remark that both end points of the arc $l$ are not the binding points $b_1$ and $b_{a(K)}$
since $1 < \alpha < \beta < a(K)$.
Then we get a lattice stick presentation of $K$ with $3a(K)-2$ sticks.

Still we have a free space to move $l$ up or down through the $z$-axis without crossing any other arcs.
Lift it up till it reaches the $z$-level $k$,
and replace the adjoined two $z$-sticks at the end points of $l$ by proper $z$-sticks.
Indeed the new $x$-stick of $l$ lies on the same line containing the $x$-stick
of the other arc with the page number $k$.
This means that we can reduce the number of sticks by $2$.
\end{proof}

\begin{lemma} \label{lem:star}
If $K$ has a star shaped arc presentation with $a(K)$ arcs,
then either $K$ is $(n+1,n)$-torus knot or $s_L(K) \leq 3 a(K) - 4$.
\end{lemma}

\begin{proof}
Recall that $a(K)$ is odd and $n = \frac{a(K)-1}{2}$.
In this presentation two arcs $c_i$ and $c_{i+n}$ adjoin at the binding point $b_{i+n}$
for $i=1, \cdots, a(K)$.

First we consider the case that the page numbers of the arcs $c_1, \cdots, c_{2n+1}$
are in the order or reverse-order of their indices $1, \cdots, 2n+1$.
Without loss of generality we assume that they are in the order.
We may assume that $c_1$ has a page number $1$ after a proper page number rotating.
For only $i=1, \dots, n$ we take the pair of arcs sharing each binding point $b_i$ and
bend this pair inside the circular disk as in Figure \ref{fig7}.
Then one can figure out that this knot is $(n+1,n)$-torus knot.

\begin{figure} [h]
\begin{center}
\includegraphics[scale=1]{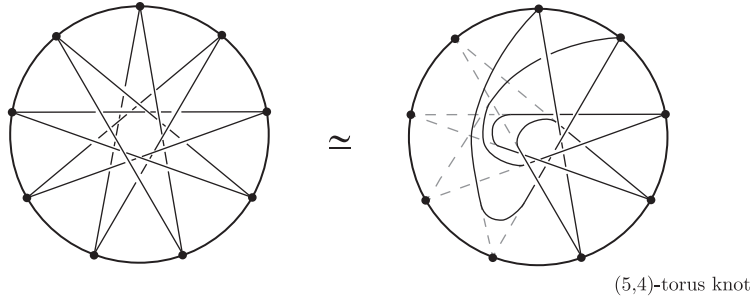}
\end{center}
\vspace{-2mm}
\caption{$(n+1,n)$-torus knot}
\label{fig7}
\end{figure}

If the page numbers are neither in the order nor reverse-order of their indices,
then one would find some binding point $b_i$ such that the difference between the page numbers
of the pair of arcs $c_i$ and $c_{i+n+1}$ adjoining at $b_i$ is neither $n$ nor $n+1$.
Thus in a dual arc presentation there exists an arc such that the difference between
the two binding indices at the end points of this arc is neither $n$ nor $n+1$.
This implies that the dual arc presentation of $K$ is not star-shaped,
so $s_L(K) \leq 3 a(K) - 4$ by Lemma \ref{lem:nonstar}
\end{proof}

\begin{lemma} \label{lem:torus}
If $K$ is $(n+1,n)$-torus knot except trefoil (so $n \geq 3$),
then $s_L(K) \leq 3 c(K) - 5$.
\end{lemma}

\begin{proof}
Let $K$ be $(n+1,n)$-torus knot for $n \geq 3$.
Note that $c(K)=(n+1)(n-1)=n^2-1$ and $a(K)=(n+1)+n=2n+1$.
So $a(K) \leq c(K) -1$.
Then $s_L(K) \leq 3 c(K) - 5$ by Lemma \ref{thm:-2}.
\end{proof}

Now we are ready to prove the main theorem.

\begin{proof}[Proof of Theorem \ref{thm:main}]
Theorem \ref{thm:main} follows directly from Lemma \ref{lem:nonstar},
\ref{lem:star}, \ref{lem:torus} and combined with Theorem \ref{thm:BP} and \ref{thm:JP}.
\end{proof}

\end{document}